\documentclass[11pt,a4paper,draft]{amsart}
\usepackage{amsfonts,amsmath,amssymb,amsthm}
\usepackage{times}
\usepackage{color}
\usepackage{pstricks}

\numberwithin{equation}{section}

\renewcommand{\epsilon}{\varepsilon}

\renewcommand{\Im}{{\ensuremath{\mathrm{Im\,}}}}
\renewcommand{\Re}{{\ensuremath{\mathrm{Re\,}}}}

\DeclareSymbolFont{SY}{U}{psy}{m}{n}
\DeclareMathSymbol{\emptyset}{\mathord}{SY}{'306}

\DeclareMathOperator{\Ran}{Ran} 
 \DeclareMathOperator{\Dom}{Dom}
\DeclareMathOperator{\sign}{sign}

\DeclareMathOperator{\spec}{spec}

\DeclareMathSymbol{\newtimes}{\mathbin}{SY}{'264}

\newcommand{\dist}{\mathrm{dist}}

\newcommand{\R}{\mathbb{R}}

\newcommand{\EE}{\mathsf{E}}

\newcommand{\fH}{\mathfrak{H}}

\newcommand{\fa}{\mathfrak{a}}
\newcommand{\fb}{\mathfrak{b}}

\newcommand{\ft}{\mathfrak{t}}

\newcommand{\fv}{\mathfrak{v}}

\newcommand{\cH}{{\mathcal H}}

\newcommand{\cV}{{\mathcal V}}

\newcommand{\ii}{\mathrm{i}}

\newtheorem{theorem}{Theorem}[section]{\bf}{\it}
\newtheorem{proposition}[theorem]{Proposition}{\bf}{\it}
\newtheorem{corollary}[theorem]{Corollary}{\bf}{\it}
\newtheorem{example}[theorem]{Example}{\it}{\rm}
{\bf}{\it}
\newtheorem{remark}[theorem]{Remark}{\it}{\rm}
\newtheorem{definition}[theorem]{Definition}{\bf}{\it}
{\bf}{\it}
{\bf}{\it}
{\bf}{\it}
\newtheorem{hypothesis}[theorem]{Hypothesis}{\bf}{\it}

\title[The TAN $2\Theta$ Theorem]{The TAN $2\Theta$ Theorem for Indefinite Quadratic Forms}

\author[L. Grubi\v{s}i\'c]{Luka Grubi\v{s}i\'c}
\address{L.~Grubi\v{s}i\'c,
Department of Mathematics, University of Zagreb, Bijeni\v{c}ka 30,
10000 Zagreb, Croatia}
\email{luka.grubisic@math.hr}

\author[V. Kostrykin]{Vadim Kostrykin}
\address{V.~Kostrykin, FB 08 - Institut f\"{u}r Mathematik,
Johannes Gutenberg-Universit\"{a}t Mainz,
Staudinger Weg 9,
D-55099 Mainz,
Germany}
\email{kostrykin@mathematik.uni-mainz.de}

\author[K. A. Makarov]{Konstantin A.~Makarov}
\address{K.~A.~Makarov, Department of Mathematics, University of
Missouri, Co\-lum\-bia, MO 65211, USA}
\email{makarovk@missouri.edu}

\author[K. Veseli\'c]{Kre\v{s}imir Veseli\'c}
\address{K.~Veseli\'c,
Fakult\"{a}t f\"{u}r Mathematik und Informatik, Fernuniversit\"{a}t Hagen, Postfach 940,
D-58084 Hagen, Germany} \email{kresimir.veselic@fernuni-hagen.de}



\copyrightinfo{2007}{L.~Grubi\v{s}i\'c, V.~Kostrykin, K.~A.~Makarov,
K.~Veseli\'c}

\begin{document}

\begin{abstract}
A version of the Davis-Kahan Tan $2\Theta$ theorem [SIAM J. Numer.
Anal. \textbf{7} (1970), 1 -- 46] for not necessarily semibounded
linear operators defined by quadratic forms is proven. This theorem generalizes a
recent result by Motovilov and Selin [Integr. Equat. Oper. Theory \textbf{56} (2006), 511 -- 542].
\end{abstract}

\maketitle

\footnotetext{This work is supported in part by the Deutsche Forschungsgemeinschaft}

\section{Introduction}

In the 1970 paper \cite{Davis:Kahan} Davis and Kahan studied the rotation of spectral subspaces for
$2\times 2$ operator matrices under off-diagonal perturbations. In particular, they proved the following result, the celebrated ``Tan $2\Theta$ theorem'': Let  $A_\pm$  be strictly positive bounded operators in Hilbert spaces $\fH_\pm$, respectively, and 
$W$  a bounded operator from $\fH_-$ to $\fH_+$.  Denote by
\begin{equation*}
A=\begin{pmatrix} A_+& 0 \\ 0 & -A_- \end{pmatrix}\quad \text{and}
\quad B=  A+V=\begin{pmatrix} A_+ & W \\ W^\ast & -A_-\end{pmatrix}
\end{equation*}
the block operator matrices with respect to the orthogonal decomposition of the Hilbert space $\fH=\fH_+\oplus \fH_-$. Then
\begin{equation}\label{tan2}
\|\tan 2\Theta\| \leq \frac{2\|V\|}{d},\qquad
\spec(\Theta)\subset [0,\pi/4),
\end{equation}
where $\Theta$ is the operator angle between the subspaces
$\Ran \EE_{A}(\R_+)$ and $\Ran \EE_{B}(\R_+)$  and 
\begin{equation*}
d=\dist (\spec (A_+) , \spec(-A_-))
\end{equation*}(see, e.g., \cite{Kostrykin:Makarov:Motovilov:2}).   

Estimate \eqref{tan2} can equivalently  be expressed as the following inequality for the norm of the difference of the orthogonal projections $P=\EE_{A}(\R_+)$ and $Q=\EE_{B}(\R_+)$:
\begin{equation}\label{tan:2:Theta}
  \|P-Q\|\leq \sin \bigg (
\frac{1}{2}\arctan \frac{2\|V\|}{d} \bigg),
\end{equation}
which, in particular, implies the estimate
\begin{equation}\label{a1}
  \|P-Q\|<\frac{\sqrt{2}}{2}.
\end{equation}

Independently of the work of Davis and Kahan, inequality \eqref{a1}
has been proven by Adamyan and Langer in \cite{Adamyan:Langer:95}, where
the operators $A_\pm$  were allowed to be semibounded. 
The case $d=0$ has been considered in the work \cite{KMM:2}
by Kostrykin, Makarov, and Motovilov. 
In particular, it was proven that there is a unique orthogonal projection $Q$ from the operator interval $[\EE_B\left ((0, \infty)\right), \EE_B\left ( [0, \infty )\right)]$ such that 
$$
  \|P-Q\|\le\frac{\sqrt{2}}{2},
$$ where $P\in\left  [\EE_A\left ((0, \infty)\right), \EE_A\left ( [0, \infty )\right)\right ]$ is the orthogonal projection onto the invariant (not necessary spectral) subspace $\cH_+\subset \cH$ of the operator $A$.
A particular case of this result has been obtained earlier by Adamyan, Langer, and Tretter,
in \cite{Adamyan:Langer:Tretter:2000a}. Recently, a version of the Tan $2\Theta$ Theorem for off-diagonal perturbations $V$ that are relatively boun\-ded with respect to the diagonal operator $A$ has been proven by Motovilov and Selin in \cite{MS:1}.

In the present work we obtain several generalizations of the aforementioned  results  assuming that the perturbation is given by an off-diagonal symmetric form.

Given a sesquilinear symmetric form $\fa$ and a self-adjoint involution $J$
such that the form $\fa_J[x,y]:=\fa[x,Jy]$ is a  positive definite 
and  $$\fa[x,Jy]=\fa[Jx,y],$$
we call a symmetric sesquilinear form $\fv$ off-diagonal with respect to the orthogonal decomposition $\fH=\fH_+ \oplus \fH_-$ with  $\fH_\pm =\Ran (I\pm J)$  if 
$$\fv[Jx,y]=-\fv[x,Jy].$$

Based  on  a close relationship between the symmetric form $\fa[x,y]+\fv[x,y]$ and the sectorial sesquilinear form $\fa[x,Jy]+\ii \fv[x,Jy]$ (cf.\ \cite{MS:1}, \cite{Veselic}), under the assumption that the off-diagonal form $\fv$ is relatively bounded with respect to the  form $\fa_J$,   
we prove 
\begin{itemize} 
\item[(i)] an analog of the First Representation Theorem for block operator matrices defined as not necessarily semibounded quadratic forms,
\item[(ii)]  a relative version of the Tan $2\Theta$ Theorem. 
\end{itemize}
 
 We also provide several versions of the relative Tan $2\Theta$ Theorem in the case where the form $\fa$ is semibounded.


\subsection*{Acknowledgments} \quad The authors thank S.~Schmitz for
useful discussions and comments. K.A.M. is indebted to the
Institute for Mathematics for its kind hospitality during his two
months stay at the Johannes Gutenberg-Universit\"{a}t Mainz in the
Summer of 2009. The work of K.A.M.\ has been supported in part by
the Deutsche Forschungsgemeinschaft and by the
Inneruniversit\"{a}ren Forschungsf\"{o}rderung of the Johannes
Gutenberg-Universit\"{a}t Mainz. L.G.\ has been supported by the
exchange program between the University of Zagreb and the Johannes
Gutenberg-Universit\"{a}t Mainz and in part by the
grant number 037-0372783-2750 of the MZO\v{S}, Croatia. K.V.~has
been supported  in part by the National Foundation of Science, Higher
Education and Technical Development of the Republic of Croatia
2007-2009.

\section{The First Representation Theorem for off-diagonal form perturbations}\label{sec:off}

To introduce the notation, it is convenient to assume the following hypothesis.

\begin{hypothesis}\label{hh1}
Let $\fa$ be a symmetric sesquilinear form on $\Dom[\fa]$  in a Hilbert space  $\fH$. Assume that $J$ is a self-adjoint involution such that 
$$
J\Dom[\fa]=\Dom[\fa].
$$
Suppose that 
$$
\fa[Jx,y]=\fa[x,Jy] \quad\text{for all } \quad  x,y\in \Dom [\fa_J]=\Dom[\fa],
$$
 and that 
 the form $\fa_J$ given by 
$$
\fa_J[x,y]=\fa[x,Jy], \quad x,y\in \Dom [\fa_J]=\Dom[\fa].
$$
is a positive definite closed form.
 Denote by $m_\pm$  the greatest lower bound of the form $\fa_J$ restricted to the subspace
 $$
 \fH_\pm =\Ran (I\pm J).
 $$ 
\end{hypothesis}

\begin{definition}
Under Hypothesis \ref{hh1}, a symmetric sesquilinear form $\fv$
on $\Dom[\fv]\supset\Dom[\fa]$ is said to be off-diagonal with respect to the orthogonal decomposition
\begin{equation*}
\fH=\fH_+\oplus\fH_-
\end{equation*}  
if
\begin{equation*}
\fv[Jx,y]=-\fv[x,Jy], \quad x,y\in \Dom[\fa].
\end{equation*}
If, in addition, 
\begin{equation}\label{xyz}
v_0:=\sup_{0\ne x\in \Dom[\fa]}\frac{|\fv[x]|}{\fa_J[x]}<\infty,
\end{equation}
the form $\fv$ is said to be an $\fa$-bounded off-diagonal form.
\end{definition}

\begin{remark}\label{rem:4} 
If $\fv$ is an off-diagonal symmetric form and 
$x=x_+ +x_-$ is a unique decomposition of an element $x\in \Dom[\fa]$ such that
$x_\pm\in \fH_\pm\cap \Dom[\fa]$, then 
\begin{equation}\label{tog}
\fv[x]=2\Re \fv[x_+,x_-],\quad x\in \Dom[\fa].
\end{equation}

Moreover, if  $v_0<\infty$, 
then
\begin{equation}\label{vad}
|\fv[x]|\leq 2v_0\sqrt{\fa_J[x_+] \fa_J[x_-]}.
\end{equation} 
\begin{proof} To prove \eqref{tog}, we use the representation
\begin{equation*} 
\fv[x]=\fv[x_+ +x_-, x_+ +x_-]=\fv[x_+]+\fv[x_-]+\fv[x_+, x_-]+\fv[x_-, x_+], \quad x\in \Dom[\fa].
\end{equation*}
Since $\fv$ is an off-diagonal form,  one obtains that
\begin{equation*}
\fv[x_+]=\fv[x_+, x_+]=\fv[Jx_+, Jx_+]=-\fv[x_+, x_+]=-\fv[x_+]=0,
\end{equation*}
and similarly $\fv[x_-]=0$. Therefore,
\begin{equation*}
\fv[x]=\fv[x_+,x_-]+\fv[x_-, x_+]=2\Re \fv[x_+,x_-],\quad x\in \Dom[\fa].
\end{equation*}

To prove \eqref{vad}, first one observes that 
$$
\fa_J[x]=\fa_J[x_+]+\fa_J[x_-]
$$
and, hence, combining \eqref{tog} and \eqref{xyz},
one gets the estimate
\begin{equation*}
|2\Re \fv[x_+,x_-]|\le v_0 \fa_J[x] = v_0 (\fa_J[x_+]+\fa_J[x_-]) \quad \text{for all}\quad x_\pm\in \fH_\pm\cap \Dom[\fa].
\end{equation*}
Hence, for any $t\ge0$ (and, therefore, for all $t\in \R$) one gets that
\begin{equation*}
 v_0 \fa_J[x_+]\,t^2-2|\Re \fv[x_+,x_-]|\,t+ v_0 \fa_J[x_-]\ge 0,
\end{equation*}
which together with \eqref{tog} implies the inequality \eqref{vad}.
\end{proof}
\end{remark}



In this setting we present an analog of the First Representation Theorem in  the off-diagonal perturbation theory.

\begin{theorem}\label{repr}
Assume Hypothesis \ref{hh1}.  Suppose that $\fv$ is  an $\fa$-bounded  off-diagonal with respect to the orthogonal decomposition $\fH=\fH_+\oplus \fH_-$  symmetric form.
On $\Dom[\fb]=\Dom[\fa]$ introduce the symmetric form
\begin{equation*}
 \fb[x,y]=\fa[x,y]+\fv[x,y], \quad x,y\in \Dom[\fb].
\end{equation*}
 Then
 \begin{itemize}
 \item[(i)] there is a unique self-adjoint operator $B$ in $\fH$ such that $\Dom(B)\subset\Dom[\fb]$ and
\begin{equation*}
\fb[x,y]= \langle x, By\rangle\quad\text{for all}\quad x\in\Dom[\fb],\quad y\in \Dom (  B).
\end{equation*}
\item[(ii)] the operator $B$ is boundedly invertible and the open interval
$(-m_-, m_+)\ni 0$ belongs to its resolvent set.
\end{itemize}
\end{theorem}

\begin{proof}
(i). Given $\mu \in (-m_-, m_+)$, on 
$\Dom[\fa_\mu]=\Dom[\fa]$   introduce  the positive closed  form $\fa_\mu$ by
\begin{equation*}
\fa_\mu[x,y]=\fa[x,Jy]-\mu \langle x, Jy\rangle, \quad x,y\in \Dom[\fa_\mu],
\end{equation*}
and denote by $\fH_{\fa_\mu}$ the Hilbert space $\Dom[\fa_\mu]$ equipped with
the inner product $\langle \cdot, \cdot \rangle_\mu=\fa_\mu[\cdot, \cdot]$.
We remark that the norms $\|\cdot\|_\mu=\sqrt{\fa_\mu[\cdot]}$
on $\fH_{\fa_\mu}=\Dom[\fa_\mu]$ are obviously equivalent.  Since $\fv$ is $\fa$-bounded, 
 one concludes  then that 
\begin{equation*}
v_\mu:=\sup_{0\ne x\in \Dom[\fa]}\frac{|\fv[x]|}{\fa_\mu[x]}<\infty,\quad \text{ for all } \mu\in(-m_-,m_+).
\end{equation*}

Along with the off-diagonal form $\fv$, introduce a dual form $\fv'$ by
\begin{equation*}
\fv'[x,y]=\ii\fv[x,Jy], \quad x,y\in \Dom[\fa].
\end{equation*}

We  claim  that 
$\fv'$ is an $\fa$-bounded off-diagonal symmetric form.
It suffices to show that 
\begin{equation*}
v_\mu=v_\mu'<\infty, \quad \mu \in (-m_-, m_+),
\end{equation*}
where
\begin{equation}\label{vaumu:bis}
 v_\mu':=\sup_{0\ne x\in \Dom[\fa]}\frac{|\fv'[x]|}{\fa_\mu[x]}.
\end{equation}

Indeed,  let $x=x_+ +x_-$ be a unique decomposition of an element $x\in \Dom[\fa]$ such that
$x_\pm\in \fH_\pm\cap \Dom[\fa]$.  By Remark \ref{rem:4},
\begin{equation*}
\fv[x]=\fv[x_+,x_-]+\fv[x_-, x_+]=2\Re \fv[x_+,x_-],\quad x\in \Dom[\fa].
\end{equation*}

In a similar way (since the form $\fv'$ is obviously off-diagonal) one gets  that
\begin{align*}
\fv'[x]&=\ii\fv[x_++x_-, J(x_++x_-)]=\ii\fv'[x_+]-\ii\fv'[x_-]-\ii\fv[x_+, x_-]+\ii\fv[x_-, x_+]
\\&=-\ii\fv[x_+,x_-]+\ii\overline{\fv[x_+, x_-]}=2\Im \fv[x_+,x_-],
 \quad x\in \Dom[\fa].
\end{align*}

Clearly, from \eqref{vaumu:bis} it follows that 
\begin{equation*}
v_\mu'= 2\sup_{0\ne x\in \Dom[\fa]}\frac{ |\Im  \fv[x_+,x_-]|}{\fa_\mu[x]}
=2\sup_{0\ne x\in \Dom[\fa]}\frac{ |\Re  \fv[x_+,x_-]|}{\fa_\mu[x]}=v_\mu,
\end{equation*}
$$\mu \in (-m_-,m_+),$$
which completes the  proof of the claim.

Next, on $\Dom[\ft_\mu]=\Dom[\fa]$ 
introduce  the sesquilinear form 
$$\ft_\mu := \fa_\mu+\ii\fv', \quad \mu \in (-m_-,m_+).$$
Since the form $\fa_\mu$ is positive definite and the form $\fv'$ is an $\fa_\mu$-bounded symmetric form, the form $\ft$
is a closed sectorial form with the vertex $0$ and semi-angle
\begin{equation}\label{tmu}
\theta_\mu=\arctan (v_\mu')=\arctan (v_\mu).
\end{equation}

Let $T_\mu$ be a unique $m$-sectorial operator associated with the form $\ft_\mu$. Introduce the operator   $$B_\mu=JT_\mu\quad \text{on}\quad  \Dom(B_\mu)=\Dom(T_\mu), \quad \mu\in(-m_-, m_+).$$
One obtains that
\begin{equation}\label{bbb}
 \begin{split}
\langle x, B_\mu y\rangle &=\langle x, JT_\mu\rangle=\langle Jx, T_\mu y\rangle
=\fa_\mu[Jx, y]+\ii\fv'[Jx,y] \\
&=\fa[x,y]-\mu\langle Jx,Jy\rangle+\ii^2\fv[Jx,Jy]
\\ &=\fa[x,y]-\mu\langle x,y\rangle+\fv[x,y],
\end{split}
\end{equation}
for all $x\in \Dom[\fa]$, $y\in \Dom(B_\mu)=\Dom(T_\mu)$.
In particular, $B_\mu$ is a symmetric operator on $\Dom(B_\mu)$, since the forms $\fa$ and $\fv$ are symmetric, and $\Dom(B_\mu)=\Dom(T_\mu)\subset \Dom[a]$.

For the real part of the form $\ft_\mu$  is  positive definite with a positive lower bound, the operator $T_\mu$ has a bounded inverse. This implies that the operator $B_\mu=JT_\mu$ has a bounded inverse and, therefore, the symmetric operator $B_\mu$ is self-adjoint on $\Dom(B_\mu)$.

As an immediate consequence, one concludes (put $\mu=0$) that  the self-adjoint operator $B:=B_0$ is associated with the symmetric form $\fb$ and that $\Dom(B)\subset \Dom[\fa]$.

To prove uniqueness, assume that $B'$ is a self-adjoint operator associated with the form $\fb$.
Then for all $x\in \Dom(B)$ and all  $y\in \Dom(B')$ one gets that
\begin{equation*}
\langle x, B'y\rangle=\fb[x,y]=\overline{\fb[y,x]}=\overline{\langle y, Bx\rangle}=\langle Bx, y\rangle,
\end{equation*}
which means that $B=(B')^*=B'$.

(ii). From \eqref{bbb} one concludes that  the self-adjoint operator
$B_\mu + \mu I$ is associated with the form $\fb$ and, hence, by the uniqueness
\begin{equation*}
B_\mu=B-\mu I\quad \text{ on }\quad\Dom(B_\mu)=\Dom(B).
\end{equation*}

Since $B_\mu$ has  a bounded inverse for all $\mu\in (m_-, m_+)$, so does $B-\mu  I$ which means that the interval $(-m_-, m_+)$ belongs to the resolvent set of the operator $B_0$.
\end{proof}

\begin{remark}\label{reg}  
In the particular case $\fv=0$, from Theorem \ref{repr} it follows that 
 there exists a unique self-adjoint operator $A$ associated with the form $\fa$. 

For a different, more constructive proof of Theorem \ref{repr}  as well as for the history of the subject we refer to our work \cite{GKMV}.  
\end{remark}

\begin{remark} 
For the part (i) of Theorem \ref{repr} to hold it is not necessary to require that the form $\fa_J$ in Hypothesis \ref{hh1} is positive definite. It is sufficient to assume that $\fa_J$ is a semi-bounded from below closed form (see, e.g., \cite{Nenciu}).
\end{remark}

\section{The Tan $2 \Theta$ Theorem}\label{sec:tan}

The main result of this work provides  a sharp upper bound for the angle between the positive spectral subspaces $\Ran \EE_A(\R_+) $ and $\Ran \EE_B(\R_+)$ of the operators $A$ and  $B$ respectively.

\begin{theorem}\label{thm:esti}
Assume Hypothesis \ref{hh1} and suppose that $\fv$ is off-diagonal with respect to the decomposition $\fH=\fH_+\oplus \fH_-$.  Let $A$ be a unique self-adjoint operator associated with the form $\fa$ and $B$   the self-adjoint operator associated with the form $\fb=\fa+\fv$ referred to in
Theorem \ref{repr}. 

Then the norm of the difference of the spectral
projections $P=\EE_{A}(\R_+)$ and $Q=\EE_B(\R_+)$ satisfies the estimate
\begin{equation*}
\|P-Q\| \leq
\sin\left(\frac{1}{2}\arctan v
\right)<\frac{\sqrt{2}}{2},
\end{equation*}
where
\begin{equation*}
v=\inf_{\mu\in(-m_-,m_+)}v_\mu=\inf_{\mu\in(-m_-,m_+)}\sup_{0\ne x\in \Dom[\fa]}\frac{|\fv[x]|}{\fa_{\mu}[x]},
\end{equation*}
with 
\begin{equation*}
\fa_\mu[x,y]=\fa[x,Jy]-\mu \langle x, Jy\rangle, \quad x,y\in \Dom[\fa_\mu] = \Dom[\fa].
\end{equation*}
\end{theorem}

The proof of Theorem \ref{thm:esti} uses the following result borrowed from \cite{Woronowicz}.

\begin{proposition}\label{lem:4:2}
Let $T$ be an m-sectorial operator of semi-angle $\theta < \pi/2$. Let $T=U|T|$ be its polar decomposition. If $U$ is unitary, then the unitary
operator $U$ is sectorial with semi-angle $\theta$.
\end{proposition}

\begin{remark}
We note that for a bounded sectorial operator  $T$ with a bounded inverse the statement is quite simple. Due to the equality
\begin{equation*}
\langle x, T x\rangle = \langle |T|^{-1/2}y, U|T|^{1/2}y\rangle = \langle
y, |T|^{-1/2} U|T|^{1/2}y\rangle,\qquad y=|T|^{1/2}x,
\end{equation*}
the operators $T$ and $|T|^{-1/2} U|T|^{1/2}$ are sectorial with the
semi-angle $\theta$. The resolvent sets of the operators $|T|^{-1/2}
U|T|^{1/2}$ and $U$ coincide. Therefore, since $U$ is unitary, it follows that $U$ is
sectorial with semi-angle $\theta$.
\end{remark}

\begin{proof}[Proof of Theorem \ref{thm:esti}]
Given $\mu\in(-m_-,m_+)$, let $T_{\mu}=U_{\mu}|T_{\mu}|$ be the polar decomposition of the sectorial
operator $T_{\mu}$ with vertex $0$ and semi-angle $\theta_\mu$, with
\begin{equation}\label{ugl}
\theta_\mu=\arctan (v_\mu)
\end{equation}
(as in the proof of Theorem \ref{repr}  (cf.~\eqref{tmu}). Since $B_{\mu}=JT_{\mu}$, one concludes that
\begin{equation*}
|T_{\mu}|=|B_{\mu}|
\quad \text{and}\quad
U_{\mu}=J^{-1}\sign(B_{\mu}).
\end{equation*}
Since $T_{\mu}$ is a sectorial operator with sem-angle $\theta_\mu$, by a result in \cite{Woronowicz} (see Proposition \ref{lem:4:2}),
the unitary operator $U_{\mu}$ is sectorial with vertex $0$ and semi-angle $\theta_\mu$ as well. Therefore,  applying  the spectral theorem  for the unitary operator $U_\mu$  from \eqref{ugl} one obtains the estimate
\begin{equation*}
\|J-\sign( B_{\mu})\|=\|I-J^{-1}\sign(B_{\mu})\|=\|I-U_{\mu}\|\le 2 \sin \left(\frac{1}{2}\arctan v_\mu\right).
\end{equation*}

Since the open interval $(-m_-,m_+)$  belongs to the resolvent
set of the operator $B=B_0$, the involution  $\sign( B_\mu)$
does not depend on $\mu\in (-m_-,m_+)$ and hence one concludes  that
\begin{equation*}
\sign( B_{\mu})=\sign( B_0)=\sign(B),\quad \mu\in (-m_-,m_+).
\end{equation*}
Therefore,
\begin{equation}\label{muo}
\|P-Q\|=\frac12 \|J-\sign( B)\|=\frac12 \|J-\sign( B_\mu)\|
\le \sin \left (\frac{1}{2}\arctan v_\mu\right)
\end{equation} and, hence,  since $\mu\in(-m_-,m_+)$ has been chosen  arbitrarily, from \eqref{muo} it follows that 
\begin{equation*}
\|P-Q\|\le \inf_{\mu\in(-m_-,m_+)} \sin \left (\frac{1}{2}\arctan v_\mu\right )\le
 \sin \left (\frac{1}{2}\arctan v \right ).
\end{equation*}

The proof is complete.
\end{proof}

As a consequence, we have the following result that can be considered a geometric variant of the  Birman-Schwinger principle for  the off-diagonal form-perturbations.

\begin{corollary}\label{cor:pi8}
Assume Hypothesis \ref{hh1} and suppose that $\fv$ is off-diagonal.
Then the form $\fa_J+\fv$ is positive definite if and only if the $a_J$-relative bound \eqref{xyz} of $\fv$ does not exceed  one. In this case 
\begin{equation*}
\|P-Q\|
\le \sin \left  (\frac{\pi}{8} \right),
\end{equation*}
where $P$ and $Q$ are the spectral projections referred to in Theorem
\ref{thm:esti}.
\end{corollary}

\begin{proof}
Since $\fv$ is an $\fa$-bounded form, one concludes that there exists a self-adjoint
bounded operator $\cV$  in the Hilbert space $\Dom[\fa]$ such that
\begin{equation*}
v[x,y]=\fa_J[x, \cV y],\quad x,y\in \Dom[\fa].
\end{equation*}
Since $\fv$ is off-diagonal, the numerical range of  $\cV$ coincides with the symmetric about the origin interval $[-\|\cV\|, \|\cV\|]$. Therefore, one can find a sequence $\{x_n\}_{n=1}^\infty$ in $\Dom[\fa]$ such that
\begin{equation*}
\lim_{n\to \infty} \frac{\fv[x_n]}{\fa_J[x_n]}=-\|\cV\|,
\end{equation*}
which proves that $\|\cV\|\le 1$ if and only if the form $\fa_J+\fv$ is positive definite. If it is the case, applying Theorem \ref{thm:esti}, one obtains the inequality
\begin{equation*}
\|P-Q\| \leq \sin \left(\frac{1}{2}\arctan \left(\|\cV\|\right )\right )\le \sin \left (\frac{\pi}{8}\right)
\end{equation*}
which completes the proof.
\end{proof}

\begin{remark}
We remark that in accordance with the Birman-Schwinger principle, for the form $\fa_J+\fv$ to have negative spectrum it is necessary that the $a_J$-relative bound $\|\cV\|$ of the perturbation $\fv$ is greater than one. As Corollary  \ref{cor:pi8} shows, in the off-diagonal perturbation theory this condition is also sufficient.
\end{remark}


\section{Two sharp estimates in the semibounded case}\label{sec:semi}

In this section we  will be dealing with the case of off-diagonal form-perturbations of 
a semi-bounded operator.

\begin{hypothesis}\label{ppp}
Assume that $A$ is a self-adjoint semi-bounded from below operator. Suppose that $A$ has a bounded inverse. Assume, in addition, that the following conditions hold:
\begin{itemize}
\item[(i)] \emph{The spectral condition.} An  open finite interval
$(\alpha, \beta)$ belongs to the resolvent  set of the operator  $A$.
We set
\begin{equation*}
\Sigma_-=\spec(A)\cap (-\infty, \alpha]\quad \text{and} \quad
\Sigma_+=\spec(A)\cap [\beta, \infty].
\end{equation*}


\item[(ii)] \emph{Boundedness.} The sesquilinear form $\fv$ is symmetric on $\Dom[\fv]\supset \Dom(|A|^{1/2})$ and
\begin{equation}\label{nach}
v:=
\sup_{0\ne x\in \Dom[\fa \,]}\frac{|\fv[x]|}{\||A|^{1/2}x\|^2}<\infty.
\end{equation}

\item[(iii)] \emph{Off-diagonality.} The sesquilinear form $\fv$ is off-diagonal with respect to
the orthogonal decomposition $\fH=\fH_+\oplus \fH_-$, with
\begin{equation*}
\fH_+=\Ran \EE_A((\beta, \infty))\quad \text{ and } \quad
\fH_-=\Ran \EE_A((-\infty, \alpha )).
\end{equation*}
That is,
\begin{equation*}
\fv[Jx,y]=-\fv[x,Jy], \quad x,y\in \Dom[\fa],
\end{equation*}
where the self-adjoint involution $J$ is given by
\begin{equation}\label{J}
J=\EE_A\left ((\beta, \infty)\right )-\EE_A\left ((-\infty, \alpha)\right ).
\end{equation}
\end{itemize}
\end{hypothesis}

Let $\fa$ be the closed form represented by the operator $A$. A direct application of Theorem \ref{repr} shows that under Hypothesis \ref{ppp} there is a unique self-adjoint boundedly invertible operator $B$ associated with the form 
$$
\fb=\fa+\fv.
$$

Under Hypothesis \ref{ppp} we distinguish two cases (see Fig.~\ref{fig:2} and \ref{fig:3}).

\begin{itemize}
\item[{\bf Case I}.]
 Assume that   $\alpha <0$ and $\beta >0$. Set \begin{equation*}
d_+=\text{dist}(\inf (\Sigma_+), 0) \quad \text{and}\quad 
d_-=\text{dist}(\inf (\Sigma_-), 0)
\end{equation*}
and suppose that $d_+>d_-$. 
\item[{\bf Case II}.]
Assume that  $\alpha,\beta >0$. Set 
 \begin{equation*}
d_+=\dist(\inf (\Sigma_+), 0)
\quad \text{and}\quad 
d_-=
\dist(\sup (\Sigma_-), 0).
\end{equation*}
\end{itemize}

As it follows from the definition of the quantities $d_\pm$,   the sum $d_-+d_+$
coincides with the distance between  the lower edges of the spectral
components $\Sigma_+$ and $\Sigma_-$ in Case I, while in Case II the difference
$d_+-d_-$ is the distance from the lower edge of $\Sigma_+$ to the  upper edge
of the spectral component $\Sigma_-$. Therefore, $d_+-d_-$ coincides
with the length of the spectral gap $(\alpha, \beta)$ of the operator $A$ in latter case.


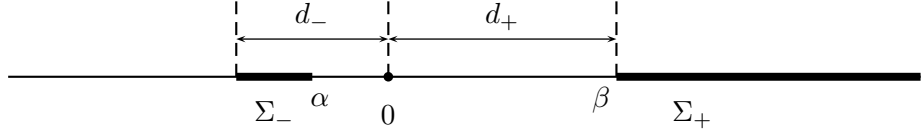
\begin{figure}[htb0]
\begin{pspicture}(12,2)
\psline(0,1)(12,1)
\psline[linewidth=3pt](8,1)(12,1)
\psline[linewidth=3pt](3,1)(4,1)
\psline[linestyle=dashed](3,1)(3,2)
\psline[linestyle=dashed](5,1)(5,2)
\psline[linestyle=dashed](8,1)(8,2)
\psline[linewidth=.5pt]{<->}(3,1.5)(5,1.5)
\psline[linewidth=.5pt]{<->}(5,1.5)(8,1.5)
\rput(3.5,0.5){$\Sigma_-$}
\rput(9,0.5){$\Sigma_+$}
\rput(5,.5){$0$}
\rput(4,1.8){$d_-$}
\rput(6.5,1.8){$d_+$}
\rput(4.1,.7){$\alpha$}
\rput(7.8,.7){$\beta$}
\psdot*[dotscale=1](5,1)
\end{pspicture}
\caption{ \label{fig:2}\small  The spectrum of the
unperturbed sign-indefinite semibounded  invertible operator $A$ in Case I. }
\end{figure}

\begin{figure}[htb0]
\begin{pspicture}(12,3)
\psline(0,1)(12,1)
\psline[linewidth=3pt](8,1)(12,1)
\psline[linewidth=3pt](6,1)(7,1)
\psline[linestyle=dashed](7,1)(7,1.5)
\psline[linestyle=dashed](5,1)(5,2.5)
\psline[linestyle=dashed](8,1)(8,2.5)
\psline[linewidth=.5pt]{<->}(5,1.5)(7,1.5)
\psline[linewidth=.5pt]{<->}(5,2.2)(8,2.2)
\rput(6.5,0.5){$\Sigma_-$}
\rput(9,0.5){$\Sigma_+$}
\rput(5,.5){$0$}
\rput(6,1.8){$d_-$}
\rput(6.5,2.5){$d_+$}
\rput(7.1,.7){$\alpha$}
\rput(7.9,.7){$\beta$}
\psdot*[dotscale=1](5,1)
\end{pspicture}\caption{ \label{fig:3}\small  The spectrum of the
unperturbed strictly positive  operator $A$ with a gap in its spectrum in Case II.}
\end{figure}
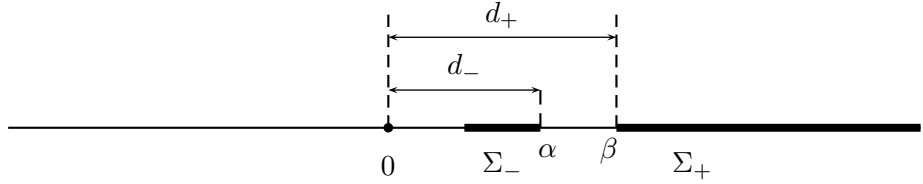

We remark that  the condition $d_+>d_-$ required in Case I,   holds only if the length of the
convex hull of negative spectrum $\Sigma_-$ of $A$
does not exceed the one of the spectral gap $ (\alpha, \beta)=\left (\sup(\Sigma_-), \inf (\Sigma_+)\right )$.

Now we are prepared to state a relative version of the Tan $2 \Theta$  Theorem in the case where the unperturbed operator is semi-bounded or even  positive.

\begin{theorem}\label{thm:esti:prime} In either  Cases I or II,
introduce the spectral projections 
\begin{equation}\label{spp}
P=\EE_A((-\infty,\alpha])\quad \text{and }\quad Q=\EE_B((-\infty, \alpha])\end{equation} 
of the operators $A$ and $B$ respectively.

Then the norm of the difference of $P$ and $Q$  satisfies the estimate
\begin{equation}\label{Davis}
\|P-Q\|\le \sin \left(\frac{1}{2}\arctan \left [2
\frac{v}{\delta}\right]\right)<\frac{\sqrt{2}}{2},
\end{equation}
where
\begin{equation}\label{delta}
\delta=\frac{1}{\sqrt{d_+d_-}}\begin{cases}
d_++d_-& \text{ in Case I},\\
d_+-d_-& \text{ in Case II},
\end{cases}
\end{equation}
and   $v$  stands for  the relative bound of the off-diagonal form $\fv$ (with respect to $\fa$)  given by \eqref{nach}. 
\end{theorem}

\begin{proof} 
We start with the remark that 
the form $\fa-\mu$, where $\fa$ is the form of $A$,  satisfies Hypothesis \ref{hh1} with $J$ given by \eqref{J}.
Set 
$$
\fa_\mu=(\fa-\mu)_J, \quad\mu \in (\alpha, \beta),
$$
that is, 
$$
\fa_\mu[x,y]=\fa[x,Jy]-\mu[x,Jy], \quad x,y\in \Dom[\fa].
$$
Notice that $\fa_\mu$ is a strictly positive closed form represented by the operators 
$JA-J\mu=|A|-\mu J$  and $ JA-\mu J=|A-\mu I|$ in Cases I and  II, respectively.

Since $\fv$ is off-diagonal, from Theorem \ref{thm:esti} it follows that 
\begin{equation}\label{first}
\|\EE_{A-\mu I}(\R_+)-\EE_{B-\mu I}(\R_+)\| \leq
\sin\left(\frac{1}{2}\arctan
v_\mu\right)\quad\text{ for all }\quad  \mu\in( \alpha, \beta),
\end{equation}
with
\begin{equation}
v_\mu=:
\sup_{0\ne x\in \Dom[\fa \,]}\frac{|\fv[x]|}{\fa_\mu[x]}.
\end{equation}

Since $\fv$ is off-diagonal, by Remark \ref{rem:4} one gets the estimate
$$
|\fv[x]|\le 2 v_0 \sqrt{\fa_0[x_+]\fa_0[x_-]}, \quad x\in \Dom [\fa],
$$
where $x=x_++x_-$ is a unique decomposition of the element $x\in \Dom [\fa]$ with 
$$x_\pm\in \fH_\pm\cap \Dom [\fa].$$
Thus, in these notations,  taking into account that  
$$
v_0=v,
$$
where $v$ is given by \eqref{nach}, 
one gets the bound
\begin{equation}\label{osnb}
v_\mu\le 2 v\sup_{0\ne x\in \Dom[\fa \,]}\frac{      \sqrt{\fa_0[x_+]\fa_0[x_-]}  }{\fa_\mu[x]}.
\end{equation}

Since $\fa_\mu$ is represented by
$JA-J\mu=|A|-\mu J$  and $ JA-\mu J=|A-\mu I|$ in Cases I and  II, respectively, 
one observes that 
\begin{equation}\label{nado}
\fa_\mu[x]=
\begin{cases}
\fa_0[x_+]-\mu\|x_+\|^2+\fa_0[x_-]+\mu\|x_-\|^2, & \text{ in  Case I,}\\
\fa_0[x_+]-\mu\|x_+\|^2-\fa_0[x_-]+\mu\|x_-\|^2, & \text{ in  Case II.}
\end{cases}
\end{equation}

Introducing the elements  $y_\pm\in \fH_\pm$,
$$y_\pm:=
\begin{cases}(|A|\mp \mu I)^{1/2} x_\pm, & \text{ in Case I},\\
\pm(A-\mu I)^{1/2} x_\pm, & \text{ in Case II},
\end{cases}
$$
and taking into account \eqref{nado},
one obtains the representation
$$
\frac{      \sqrt{\fa_0[x_+]\fa_0[x_-]}  }{\fa_\mu[x]}=
 \frac{\||A|^{1/2}(|A|- \mu I)^{-1/2}y_+\|\, \||A|^{1/2}(-A+ \mu I)^{-1/2}y_-\|}
{\|y_+\|^2+\|y_-\|^2},
$$
valid in both Cases I and II. 
Using the elementary inequality
$$
\|y_+\|\, \|y_-\|\le \frac12\left  (\|y_+\|^2+\|y_-\|^2\right ), 
$$
one arrives at the following bound
\begin{equation}\label{eins}
\frac{      \sqrt{\fa_0[x_+]\fa_0[x_-]}  }{\fa_\mu[x]}\le\frac12
 \||A|^{1/2}(|A|- \mu I)^{-1/2}|_{\fH_+}\|\cdot\||A|^{1/2}(-A+ \mu I)^{-1/2}|_{\fH_-}\|.
\end{equation}

It is easy to see that
\begin{equation}\label{zwei}
\||A|^{1/2} (|A|-\mu I)^{-1/2}|_{\fH_+}\|  \leq
\frac{\sqrt{d_+}}{\sqrt{d_+ -\mu}} \quad \mu\in (\alpha, \beta),\quad \text{ in Cases I and II},  
\end{equation}
while
\begin{equation}\label{drei}
\||A|^{1/2} (-A+\mu I)^{-1/2}|_{\fH_-}\| \leq
\begin{cases}
\frac{\sqrt{d_-}}{\sqrt{d_-+\mu}},& \mu\in (0, \beta),\quad \text{ in Case I},\\
\frac{\sqrt{d_-}}{\sqrt{\mu
-d_-}},& \mu\in (\alpha, \beta),\quad \text{ in Case  II}.
\end{cases}
\end{equation} 
Choosing $\mu=\frac{d_+-d_-}{2}>0$ in Case I  (recall that $d_+>d_-$ by the hypothesis) 
 and $\mu=\frac{d_++d_-}{2}$ in Case II, and combining \eqref{eins}, \eqref{zwei}, \eqref{drei},  one gets the estimates
$$
\frac{      \sqrt{\fa_0[x_+]\fa_0[x_-]}  }{\fa_{\frac{d_+-d_-}{2}}[x]}\le 
\frac{\sqrt{d_+d_-}}{d_++d_-}\quad\text{in Case I}
$$
and
$$
\frac{      \sqrt{\fa_0[x_+]\fa_0[x_-]}  }{\fa_{\frac{d_++d_-}{2}}[x]}\le 
\frac{\sqrt{d_++d_-}}{d_+-d_-}\quad \text{in Case II}
.
$$
Hence, from \eqref{osnb} it follows that 

\begin{equation*}
v_{\frac{d_+-d_-}{2}}\le 2v
\frac{\sqrt{d_+d_-}}{d_++d_-}\quad\text{in Case I}
\end{equation*}
and 
\begin{equation*}
v_{\frac{d_++d_-}{2}}\le 2v
\frac{\sqrt{d_+d_-}}{d_+-d_-}\quad \text{in Case II}.
\end{equation*}

Applying \eqref{first}, one gets the norm estimates

\begin{equation}\label{first1}
\|\EE_{A-\frac{d_+-d_-}{2} I}(\R_+)-\EE_{B-\frac{d_+-d_-}{2} I}(\R_+)\|
\leq
\sin\left(\frac{1}{2}\arctan\left[2
\frac{\sqrt{d_+d_-}}{d_++d_-}
v\right]\right)
\end{equation}
in Case I and 
\begin{equation}\label{first2}
\|\EE_{A-\frac{d_++d_-}{2} I}(\R_+)-\EE_{B-\frac{d_++d_-}{2} I}(\R_+)\|
\leq
\sin\left(\frac{1}{2}\arctan\left[2
\frac{\sqrt{d_+d_-}}{d_+-d_-}
v\right]\right)
\end{equation}
in Case II.
In remains to observe that $\|P-Q\|$, where the spectral projections $P$ and $Q$ are given by \eqref{spp},  coincides with the left hand side of \eqref{first1}   and \eqref{first2}   in Case I and   Case II, respectively.

The proof is complete.
\end{proof}

\begin{remark} 
We remark that 
the quantity $\delta$ given by \eqref{delta}  coincides with   the \emph{relative distance}
(with respect to the origin) 
between the lower edges of the spectral components $\Sigma_+$ and $\Sigma_-$ in Case I and 
it has the meaning of the \emph{relative length}  (with respect to the origin) of the spectral gap $(d_-, d_+)$ in Case II.

For the further properties of the relative distance and various relative perturbation bounds we
refer to the paper \cite{Li} and references quoted therein.

We also remark that in Case II, i.e., in the case of a positive operator $A$,
the bound \eqref{Davis} directly improves a result  obtained in \cite{Luka},
\emph{the relative $\sin\Theta$ Theorem}, that in the present notations is of the form
\begin{equation*}
\|P-Q\|\le \frac{v}{\delta}.
\end{equation*}
\end{remark}

We conclude our exposition with considering an example of a $2\times 2$ numerical matrix that shows that the main results obtained above are sharp.

\begin{example}\label{exam}
\emph{Let $\fH$ be the  two-dimensional Hilbert space $\fH=\mathbb{C}^2$,
$\alpha<\beta$ and $w\in \mathbb{C}$. }

\emph{We set }
\begin{equation*}
A=\begin{pmatrix} \beta &0 \\ 0& \alpha
\end{pmatrix}
, \quad
V=\begin{pmatrix}
0& w\\
w^*&0\end{pmatrix}\quad \text{ \emph{and} }\quad J=\begin{pmatrix} 1&0\\0&-1
\end{pmatrix}.
\end{equation*}

\emph{Let  $\fv$ be the symmetric form represented by (the operator) $V$.}

\emph{Clearly, the form $\fv$ satisfy Hypothesis \ref{ppp} with the relative bound $v$ given by
\begin{equation*}
v=\frac{|w|}
{\sqrt{|\alpha\beta|}},
\end{equation*}
provided that  $\alpha, \beta\ne0$.
Since $VJ=-JV$, the form $\fv$ is off-diagonal
with respect to the  orthogonal decomposition $\fH=\fH_+\oplus\fH_-$.}

\end{example}

In order to illustrate our results, 
denote by $B$ the self-adjoint matrix associated with the form $\fa+\fv$, that is,
\begin{equation*}
B=A+V=\begin{pmatrix} \beta & w \\ w^* &\alpha\end{pmatrix}.
\end{equation*}

Denote by $P$ the orthogonal projection associated with the  eigenvalue $\alpha$
of the matrix $A$, and by $Q$ the one associated with the lower eigenvalue of the matrix $B$.

It is well know (and easy to see) that
the classical Davis-Kahan Tan $2\Theta$ theorem \eqref{tan:2:Theta} 
is exact in the case of $2\times2$ numerical matrices. In particular,
the norm of the difference of $P$ and $Q$ can be computed explicitly 
\begin{equation}\label{sharp}
\|P-Q\|=\sin \left(\frac12 \arctan \left[\frac{2|w|}{\beta-\alpha}\right]\right).
\end{equation}

Since, in the case in question, 
\begin{equation}\label{susu2}
v_\mu=\sup_{0\ne x\in \Dom[\fa \,]}\frac{|\fv[x]|}{\fa_\mu[x]}=\frac{|w|}
{\sqrt{(\beta-\mu)(\mu-\alpha)}},\quad \mu \in (\alpha, \beta),
\end{equation}
from \eqref{susu2} it follows that
\begin{equation*}
\inf_{\mu\in(\alpha, \beta)}v_\mu=\frac{2|w|}{\beta-\alpha}
\end{equation*} 
(with the infimum attained at the point $\mu =\frac{\alpha+\beta}{2}$).

Therefore,  the result of the relative $\tan 2 \Theta$
Theorem \ref{thm:esti} is sharp.

It  is easy to see that if  $\alpha < 0<\beta$ (Case I), then
the equality \eqref{sharp} can also be rewritten in the form
\begin{equation}\label{zz}
\|P-Q\|=
\sin \left(\frac12 \arctan \left[2\frac{\sqrt{d_+d_-}}{d_++d_-}v\right]\right),
\end{equation}
where $d_+=\beta$, $d_-=-\alpha$ and $v=\frac{|w|}{\sqrt{|\alpha|\beta}}$.

If $0<\alpha <\beta$ (Case II),
the equality \eqref{sharp} can be rewritten as
\begin{equation}\label{zzz}
\|P-Q\|=
\sin \left(\frac12 \arctan \left[2\frac{\sqrt{d_+d_-}}{d_+-d_-}v\right]\right),
\end{equation}
with  $d_+=\beta$, $d_-=\alpha$, and $v=\frac{|w|}{\sqrt{\alpha\beta}}$.

The representations \eqref{zz} and \eqref{zzz} show that the estimate \eqref{Davis}
 becomes equality in the case of $2\times 2$ numerical matrices
and, therefore, the results of Theorem \ref{thm:esti:prime}  are sharp.


\end{document}